\documentclass[11pt]{article}
\usepackage{amssymb}
\usepackage{amsthm}
\usepackage{latexsym}
\usepackage{makecell}
\usepackage{tikz}
\usepackage{multirow,bigdelim}
\usepackage{amsmath}
\usepackage{mathabx}
\usepackage{bbm}
\usepackage[all]{xy}
\usepackage{titling}
\thanksmarkseries{arabic}
\usepackage{hyperref}
\usepackage{accents}
\hypersetup{
  colorlinks   = true, 
  urlcolor     = blue, 
  linkcolor    = blue, 
  citecolor   = red 
}
\theoremstyle{definition}
\newtheorem{theo}{Theorem}[section]

\newtheorem{cor}[theo]{Corollary}
\newtheorem{lem}[theo]{Lemma}
\newtheorem{notation}[theo]{Notation}
\newtheorem{defi}[theo]{Definition}
\newtheorem{defiprop}[theo]{Definition/Proposition}

\newtheorem{rem}[theo]{Remark}

\newtheorem{prob}[theo]{Problem}
\newtheorem{conjecture}[theo]{Conjecture}
\numberwithin{equation}{section}
\newcommand{\QED}{\mbox{}\hfill$\Box$\medskip}

\newcommand*\hexbrace[2]{%
  \underset{#2}{\underbrace{\rule{#1}{0pt}}}}         
\let\oldemptyset\emptyset

\newcommand{\Z}{{\mathbb Z}}
\newcommand{\Q}{{\mathbb Q}}

\newcommand{\R}{{\mathbb R}}

\newcommand{\cC}{{\mathcal C}}

\newcommand{\cP}{{\mathcal P}}

\newcommand{\cB}{{\mathcal B}}

\newcommand{\cW}{{\mathcal W}}

\newcommand{\Zkn}{\mbox{$\Z_{n\tau}$}}

\newcommand{\hwt}{\mbox{${\rm wt}_{\rm H}$}}
\newcommand{\rk}{\mbox{${\rm rk}$}}

\newcommand{\rate}{\mbox{\rm rate}}
\newcommand{\im}{\mbox{\rm im}\,}
\newcommand{\rs}{\mbox{\rm rs}\,}

\newcommand{\supp}{\mbox{\rm supp}\,}

\renewcommand{\mod}{\,\mbox{\rm mod}\,}

\newcommand{\sbt}{\raisebox{.2ex}{\mbox{$\scriptscriptstyle\bullet\,$}}}
\newcommand{\T}{\mbox{$\!^{\,\sf T}$}}
\newcommand{\w}{\raisebox{-.3ex}{\mbox{$\!^{\tiny\vee}$}}}

\newcommand{\group}[1]{\mbox{$\langle{#1}\rangle$}}

\newcommand{\wh}[1]{\mbox{$\widehat{#1}$}}
\newcommand{\ov}[1]{\mbox{$\overline{#1}$}}
\newcommand{\vect}[1]{\mbox{\textbf{#1}}}
\newcommand{\GL}{\mathrm{GL}}

\newcommand{\dist}{\textup{dist}}

\newcommand{\one}{\mathbbm{1}}
\newcommand{\conv}[1]{\mathrm{conv}\left(#1\right)}
\newcommand{\aff}[1]{\mathrm{aff}\left(#1\right)}
\def\ss{\textsuperscript}
\newcommand{\Vol}{\operatorname{Vol}}
\newcommand{\he}{\text{ht}}

\newcounter{alp}
\newcounter{ara}
\newcounter{rom}

\newenvironment{arabiclist}{\begin{list}{(\arabic{ara})\hfill}{\usecounter{ara}
     \topsep0.4ex \labelwidth.6cm \leftmargin.6cm \labelsep0cm
     \rightmargin0cm \parsep0ex \itemsep0ex}}{\end{list}}
\begin{document}
\title{Laplacian Simplices II: A Coding Theoretic Approach}
\date{\today}
\author{Marie Meyer\thanks{MM is with the Department of Computer and Mathematical Sciences, Lewis University, Romeoville IL 60446, USA; mmeyer2@lewisu.edu.} \ and Tefjol Pllaha\thanks{TP is with the Department of Mathematics, University of Kentucky, Lexington KY 40506-0027, USA; tefjol.pllaha@uky.edu.}}

\maketitle

{\bf Abstract:} This paper further investigates \emph{Laplacian simplices}. 
A construction by Braun and the first author associates to a simple connected graph $G$ a simplex $\cP_G$ whose vertices are the rows of the Laplacian matrix of $G$. 
In this paper we associate to a reflexive $\cP_G$ a duality-preserving linear code $\cC(\cP_G)$.
This new perspective allows us to build upon previous results relating graphical properties of $G$ to properties of the polytope $\cP_G$. 
In particular, we make progress towards a graphical characterization of reflexive $\cP_G$ using techniques from Ehrhart theory.
We provide a systematic investigation of $\cC(\cP_G)$ for cycles, complete graphs, and graphs with a prime number of vertices. We construct an asymptotically good family of MDS codes. In addition, we show that any rational rate is achievable by such construction.

\textbf{Keywords:} lattice polytope, Ehrhart theory, Laplacian simplex, linear codes, duality.

\textbf{MSC (2010):} 05C50, 52B20, 94B05



\section{Introduction}
Laplacian simplices were introduced by Braun and the first author in~\cite{BM17} as a way to associate a polytope to a simple connected graph $G$ through its Laplacian matrix. 
This construction is analogous to that of the \emph{edge polytope}, the convex hull of the columns of the unsigned vertex-edge incidence matrix of a graph, which has been studied in detail over the past several decades; see~\cite{HibiEhrhartEdge,hibiohsugiedgepolytope,TranZiegler,VillarrealEdgePolytopes}. 
Properties of the Laplacian simplex associated to $G$, denoted $\cP_G$, are directly related to properties of $G$. For instance, the normalized volume of $\cP_G$ is equal to the product of the number of vertices and the number of spanning trees of $G$~\cite[Prop. 3.4]{BM17}. 
Certain families of graphs exhibit certain behaviors of $\cP_G$.
For example, if $G$ is a tree or complete graph, then $\cP_G$ is reflexive and satisfies the integer decomposition property. 
In the case of a cycle, $\cP_G$ is reflexive if and only if the cycle has an odd number of vertices.  
The interplay between the graphical structure of $G$ and the geometric structure of $\cP_G$ has already produced interesting results related to reflexivity, the $h^*$-vector, and the integer decomposition property of the polytope~\cite{BM17}. 
A natural generalization of $\cP_G$ is examined in~\cite{digraph}, by considering the Laplacian simplex associated to a digraph $D$. Certain families of these Laplacian polytopes are realized as $q$-simplices and weighted projective spaces.
Thus, although recently introduced, Laplacian simplices are combinatorial objects with interesting connections to other families of simplices. 

This paper extends the initial investigation of $\cP_G$ through the lens of coding theory. The use of lattice polytopes to generate linear codes is a recently developed technique.  
Batyrev and Hofscheier first used linear codes when classifying lattice simplices with a specific binomial $h^*$-polynomials. They associate the additive group of lattice points in the fundamental parallelepiped with a multiplicative group from which yields a linear code~\cite{BH13}. 
This technique is also used in~\cite{HNT15} as a tool in the characterization of lattice simplices with palindromic trinomial $h^*$-polynomial. 
We apply this technique to Laplacian simplices. By examining the lattice points in the fundamental parallelepiped, i.e. codewords, we improve results pertaining to reflexivity, duality, the Ehrhart polynomial, unimodality, and constructions of reflexive Laplacian simplices.  

Our contribution is twofold: an extension of known results of $\cP_G$ found in~\cite{BM17} and an analysis of codes arising from reflexive $\cP_G$. In Section~\ref{sec:Basics}, we introduce lattice polytopes and highlight fundamental techniques from Ehrhart theory that will be frequently used throughout the paper. 
This is followed by a brief description of basic notions in coding theory over integer residue rings.  
In Section~\ref{sec:Laplacian}, we start with a general introduction to Laplacian simplices and an outline of our approach to study lattice points in the fundamental parallelepiped, denoted  $\Pi(\cP_G)$. 
Then we consider only $\cP_G$ which are reflexive.
We provide a vertex description of the dual of a Laplacian simplex and rediscover a known characterization of reflexive $\cP_G$. 
The lattice points in the fundamental parallelepiped of a Laplacian simplex have a particularly nice structure. This allows us to associate to a reflexive $\cP_G$ a linear code $\cC(\cP_G)$ over integer residue rings. We show that this association preserves duality, that is, lattice points in the fundamental parallelepiped of the dual simplex correspond to the dual of $\cC(\cP_G)$. 
The latter half of Section~\ref{sec:Laplacian} presents graph operations on $G$ which preserve the reflexivity of $\cP_G$. We classify $\Pi(\cP_G)$ for whiskered graphs, denoted $\cW(G)$, and bridges of graphs. Whenever $\cP_G$ is reflexive, we show the $h^*$-polynomial of $\cP_{\cW(G)}$ can be written in terms of that of $\cP_G$. The operation of whiskering the whiskers of $\cW(G)$ is generalized to reveal further reflexive families.
Next we consider a new construction of starring the whiskers of a whiskered complete graph. The resulting simplex is reflexive, and we classify its fundamental parallelepiped lattice points. 
Finally we show the bridge of two graphs on the same vertex set whose Laplacian simplices are reflexive yields a reflexive Laplacian simplex, an extension of \cite[Thm. 3.14]{BM17}. This result generalizes to include bridges of any number of graphs. 
In Section~\ref{sec:codes}, we provide an analysis of the linear codes associated to families of reflexive Laplacian simplices. We pay special attention to simple connected graphs with a prime number of vertices. Among other results, we show that starring the whiskers of a whiskered complete graph yields asymptotically good linear codes. We end the paper with some conclusions and future research.

\section{Preliminaries}\label{sec:Basics}
\subsection{Basic Notions of Lattice Polytopes}
A \emph{polytope} $\cP$ is the \emph{convex hull} of finitely many points $\{\vect{v}_1, \ldots, \vect{v}_n\} \subset \R^d$, that is
\begin{equation}\label{e-VP}
\cP=\conv{\vect{v}_1, \ldots, \vect{v}_n}=\left\{ \sum_{i=1}^n\lambda_i \vect{v}_i\,\,\middle|\,\, \sum_{i=1}^n\lambda_i=1,\,\lambda_i \ge 0\right\},
\end{equation}
and $\vect{v}_i$'s are the vertices of $\cP$. 
We focus exclusively on \emph{lattice polytopes}, i.e. polytopes whose vertices have integer entries.
The dimension of $\cP$ is the dimension of the affine space
\[
\aff{\cP} := \{\vect{p} + \alpha (\vect{p}'-\vect{p}) \mid \vect{p},\vect{p}' \in \cP, \alpha \in \R \}\supset \cP.
\]
We will think of a point $\vect{p} \in \R^d$ as a row\footnote[2]{Row vectors are handier because we will often append a 1 to our vectors. Also row vectors are customary in coding theory.} vector. If $n = d+1$ and $\dim \cP = d$ then $\cP$ is called a \emph{$d$-simplex}. In other words, the convex hull of $d+1$ points in $\R^d$ is a $d$-simplex if it is full-dimensional. We stack the vertices of $\cP$ as the rows of a matrix $V\in \Z^{n\times d}$, called the \emph{vertex matrix of $\cP$}. Thus, $\cP$ is a convex subset of the row space of the vertex matrix. For this reason \eqref{e-VP} is called the \emph{vertex description} of $\cP$. For a matrix $M$ we will denote $\conv{M} \subset \rs(M)$ the convex hull of the rows of $M$. For any permutation matrix $P$ we have $\conv{M} = \conv{PM}$ and thus any relabeling of the vertices does not affect the polytope. In this context, relabeling the vertices corresponds to permuting the rows of the vertex matrix. 
Two lattice polytopes $\cP=\conv{\vect{v}_1, \ldots, \vect{v}_n}$ and $\cP'=\conv{\vect{v}_1', \ldots, \vect{v}_n'}$ are \emph{unimodularly equivalent} if there exists an invertible matrix $U\in \GL_d(\Z)$ and a vector $\vect{b}\in \Z^d$ such that $\vect{v}_i' = \vect{v}_iU+\vect{b}$.
Consequently, results for $\cP$ also hold for unimodularly equivalent $\cP'$. 
Every polytope has a vertex description as well as a \emph{hyperplane description}, i.e, writing $\cP$ as the intersection of finitely many hyperplanes.
In particular, we can always write $\cP = \{\vect{x} \in \R^{d} \mid A\vect{x}\T \le \vect{b}\T\}$.


\begin{rem}\label{R-hyperplane}
Let $A_i$ denote the $i\ss{th}$ row of $A$.
To prove $\cP = \{\vect{x}\in \R^{d} \mid A\vect{x}\T \le \vect{b}\}$ is the hyperplane description of the lattice simplex $\cP=\conv{\vect{v}_1, \ldots, \vect{v}_{d+1}}$, it suffices to show that for each $\vect{v}_i$, 
\[
A_i\vect{v}_i\T < b_i \text{ and } A_j \vect{v}_i = b_j \text{ for all } j \ne i \in [d+1].
\] 


\end{rem}

The class of lattice polytopes we consider will be full-dimensional simplices containing the origin in their interior.
The \emph{dual polytope} of a full-dimensional polytope $\cP$ which contains $\vect{0} \in \R^d$ is 
\[
\cP^{\w} := \{ \vect{x} \in \R^d \mid \vect{x} \,\vect{y}\T \le 1 \text{ for all } \vect{y} \in \cP\} \,.
\] 

\begin{rem}\label{R-dual} The vertex description of a polytope yields a hyperplane description of its dual polytope.  
If $\cP = \conv{\vect{v}_1, \ldots, \vect{v}_n}$ with $\vect{0} \in \cP$, then we write
\[
\cP^{\w} = \{ \vect{x} \in \R^d \mid V\vect{x}\T \le \one \},
\]
where $\{\vect{v}_i\}_{i=1}^n$ are the rows of $V$; see~\cite[Thm. 2.11]{ZieglerLectures}.
\end{rem}

Regarding symmetry, reflexive polytopes are a particularly interesting family of polytopes.
A lattice polytope $\cP$ is called \emph{reflexive} if it contains the origin in its interior, and its dual $\mathcal P^{\w}$ is a lattice polytope. 
Consequently, reflexive polytopes arise in pairs; $\cP$ is reflexive iff $\cP^{\w}$ is.
From the description of $\cP^{\w}$ given in Remark~\ref{R-dual}, it follows that $\cP=\conv{\vect{v}_1, \ldots,\vect{v}_n}$ is reflexive iff we can write $\cP=\{\vect{x}\in \R^d \mid A\vect{x}\T \le \one\}$ for some integer matrix $A$.

We count the number of lattice points in the $t\ss{th}$ dilate of $\cP$, written $t \cP := \{t\vect{p} \mid \vect{p} \in \cP\}$ for some $t \in \Z_{>0}$, with the function $L_{\cP}(t) := | t \cP \cap \Z^d |$. It is shown that $L_{\cP}(t)$ is a polynomial of degree $d=\dim(\cP)$ and is referred to as the \emph{Ehrhart polynomial} of $\cP$~\cite{Ehrhart}. 
Note that two unimodularly equivalent polytopes have identical Ehrhart polynomials. 
The generating function of the Ehrhart polynomial is called the \emph{Ehrhart series}. 
It is well-known \cite{BeckRobinsCCD} that the Ehrhart series of a lattice polytope is a rational function of the form
\[
\text{Ehr}_{\cP}(z) = 1 + \sum_{t\ge1} L_{\mathcal P}(t)z^t = \frac{h_d^*z^d + h_{d-1}^*z^{d-1} + \cdots + h_1^*z + h_0^*}{(1-z)^{d+1}}
\] 
where the numerator $h_d^*z^d + h_{d-1}^*z^{d-1} + \cdots + h_1^*z + h_0^*$ is a polynomial of degree at most $d=\dim(\cP)$ with nonnegative integer coefficients and $h_0^*=1$.
This is called the \emph{$h^*$-polynomial} of $\cP$ and is an important invariant as it conveys much information about $\cP$. 
The leading coefficient $h_d^*$ is the number of interior lattice points of $\cP$, and the linear coefficient $h_1^*$ is $|\cP \cap \Z^d|-d-1$.
The sum $\sum_{i=0}^d h_i^*$ is the \emph{normalized volume} of $\cP$, denoted by $\Vol(\cP)$; it is equal to $d!$ times the usual Euclidean volume of $\cP$.  
We extract the coefficients to form the \emph{$h^*$-vector} of $\cP$, $h^*(\mathcal P) := (h_0^*, h_1^*, \ldots , h_d^*)$.
For the case of symmetric $h^*$-vectors, Hibi established the following connection to reflexive polytopes.
\begin{theo}[\cite{Hibi1}]\label{T-Hibi}
A $d$-dimensional lattice polytope $\mathcal P \subset \R^d$ containing the origin in its interior is reflexive if and only if $h^*(\mathcal P)$ satisfies $h_i^* = h_{d-i}^*$ for $0 \le i \le \lfloor \frac{d}{2} \rfloor.$
\end{theo}

A vector $\vect{x} = (x_0, x_1, \ldots , x_n)$ is \emph{unimodal} if there exists a $j \in [n]$ such that $x_i \le x_{i+1}$ for all $0 \le i < j$ and $x_k\ge x_{k+1}$ for all $j \le k < n$.
The cause of unimodality among $h^*$-vectors is unknown but of much interest.  
A long standing conjecture asserts that being reflexive and satisfying the integer decomposition property, see~\cite{hibiohsugiconj}, is a sufficient condition for a lattice polytope to have a unimodal $h^*$-vector. 
In~\cite{P08}, the author provides explicit construction of reflexive lattice polytopes with non-unimodal $h^*$-vectors. 
However, it is unknown whether or not IDP alone is a sufficient condition for unimodality.

Simplices play a special role in Ehrhart theory, as there is an alternative and practical method for computing their $h^*$-vectors.
The \emph{cone over} $\cP=\conv{\vect{v}_1, \ldots, \vect{v}_{n}}$ is defined as
\[
\text{cone}(\cP) = \left\{ \sum_{i=1}^n\lambda_i(\vect{v}_i, 1)\,\,\middle|\,\, \lambda_i \ge 0\right\} \subseteq \R^{d+1},
\]
where $(\vect{v}_i, 1)$ is the vertex $\vect{v}_i \in \cP \subset \R^d$ with an appended $1$. The \emph{fundamental parallelepiped} of $\cP$ is the subset of $\text{cone}(\cP)$ defined by 
\[
\Pi(\cP) := \left\{ \sum_{i=1}^{n} \lambda_i (\vect{v}_i, 1) \,\,\middle|\,\, 0 \le \lambda_i < 1  \right\}.
\]

\begin{lem}[\cite{BeckRobinsCCD}]\label{lem:fpp}
Given a lattice simplex $\cP = \conv{\vect{v}_1, \ldots, \vect{v}_{d+1}}$, we have
\[
h_i^*(\cP) = | \Pi(\cP) \cap \{ \vect{x} \in \Z^{d+1} \mid x_{d+1}=i \}|.
\]
Further, $\sum_{i=0}^d h^*_i(\cP) = |\Pi(\cP)\cap \Z^{d+1}|=\Vol(\cP)$, which is also equal to the determinant of the matrix whose $i$\ss{th} row is given by $(\vect{v}_i,1)$.
\end{lem}
In~\cite{BH13}, Batyrev and Hofscheier introduce a connection between lattice simplices, finite abelian groups, and coding theory. 
We build on this work in the context of Laplacian simplices.
To a simplex $\mathcal P = \conv{\vect{v}_1, \ldots, \vect{v}_{d+1}}$ we associate a finite abelian subgroup of $(\Q/\Z)^{d+1}$.
The main idea is to keep track of the linear combinations of vertices which yield fundamental parallelepiped points. Define
\begin{equation}
\Lambda(\cP) := \left\{\lambda = (\lambda_1,\ldots,\lambda_{d+1})\,\, \middle| \,\, \sum_{i=1}^{d+1}\lambda_i(\vect{v}_i, 1) \in \Pi(\cP)\cap \Z^{d+1}\right\}.
\end{equation}
Since $\cP$ is a lattice polytope, for $\lambda\in \Lambda(\cP)$ we have $\lambda_i\in \Q\cap [0,1).$ Thus addition modulo $\Z$ in $\Lambda(\cP)$ is given by
\begin{equation}\label{e-add}
(\lambda_1,\ldots,\lambda_{d+1}) + (\lambda_1',\ldots,\lambda_{d+1}') = (\{\lambda_1+\lambda_1'\},\ldots,\{\lambda_{d+1}+\lambda_{d+1}'\}),
\end{equation}
where $\{\sbt\}$ denotes the fractional part of a number. It follows directly by the definition that 
\begin{equation}\label{e-lfp}
|\Lambda(\cP)| = |\Pi(\cP)\cap\Z^{d+1}| = \Vol(\cP).
\end{equation}
For $\lambda\in \Lambda(\cP)$ the quantity $\he(\lambda):=\sum_{i}^{d+1}\lambda_i\in \Z$ is called the \emph{height of $\lambda$} whereas $\supp(\lambda) : = \{i\mid \lambda_i \neq 0\}$ is called the \emph{support} of $\lambda$. Then, Lemma \ref{lem:fpp} reads as
\begin{equation}\label{e-h*}
h^*_i(\cP) = |\{\lambda \in \Lambda(\cP)\mid \he(\lambda) = i\}|,
\end{equation}
and the reader will verify that
\begin{equation}\label{e-ht}
\he(\lambda) + \he(-\lambda) = |\supp(\lambda)|. 
\end{equation}
Equation \eqref{e-ht} constitutes an obvious connection with coding theory.

\subsection{Basic Notions in Coding Theory}
For natural numbers $m$ and $n$, denote $\Z_m:=\Z/m\Z$ the integer residue ring modulo $m$ and $\Z_m^n$ the direct product of $n$ copies of $\Z_m$. A submodule $\cC \subset\Z_m^n$ is called \emph{linear code of length $n$ over $\Z_m$}. Elements of $\cC$ are called \emph{codewords} and $\Z_m$ is called the \emph{alphabet}. The \emph{Hamming weight} of a codeword $c = (c_1,\ldots ,c_n)$ is given by 
\begin{equation}\label{e-hwt}
\hwt(c):=|\supp(c)|,
\end{equation}
where again $\supp(c) = \{i\mid c_i \neq 0\}$. The \emph{minimum distance} of $\cC$ is then given by
\begin{equation}
\dist(\cC):=\min\{\hwt(c)\mid c\in \cC - \{0\}\}.
\end{equation}
Next, the \emph{dual code} of $\cC$ is given by
\begin{equation}
\cC^\perp:=\{x\in \Z_m^n \mid x\cdot c = 0 \text{ for all } c\in \cC\},
\end{equation}
where $x\cdot c = xc\T=\sum^n_{i=1}x_ic_i$ is the standard dot product in $\Z_m^n$. We say $\cC$ is self-orthogonal (resp., self-dual) if $\cC\subseteq \cC^\perp$ (resp., $\cC = \cC^\perp$). 

The following easily verifiable result will be a crucial counting tool.
\begin{rem}\label{R-CCt}
Let $\cC \subseteq \Z_m^n$ be a linear code. Then $\cC^\perp \cong \Z_m^n/\cC$. As a consequence $|\cC|\!\cdot \!|\cC^\perp| = |\Z_m^n| = m^n$.
\end{rem}
Two linear codes $\cC,\,\cC'$ are called \emph{monomially equivalent} if there exist a permutation $\sigma\in S_n$ and units $u_i\in \Z_n^*$ such that for all $c = (c_1,\ldots,c_n)\in \cC$ we have $(u_1c_{\sigma(1)},\ldots,u_nc_{\sigma(n)})\in \cC'$. If $u_i = 1$ for all $i$ then $\cC$ and $\cC'$ are called \emph{permutation equivalent}. $\cC$ is called \emph{cyclic} if $(c_n,c_1,c_2\ldots,c_{n-1})\in\cC$ for all $c\in \cC$. 
\begin{rem}\label{R-code2}
If $m = p$ is a prime then $\Z_p^n$ is an $\Z_p$-vector space of dimension $n$. Let $\cC\subseteq \Z_p^n$ be linear code of dimension $k$, and let $\{c_1,\ldots, c_k\}$ be a basis of $\cC$. Then $\cC$ is the rowspace of the matrix $A$ whose $i\ss{th}$ row is $c_i$. Such a matrix is called \emph{generating matrix}. A full-dimensional matrix $H$ such that $GH\T=0$ is called \emph{parity-check matrix}. By performing row and column operations we can bring $A$ into the form $[I_k\mid A']$. Such a matrix is said to be in \emph{standard form}. Then $\cC':=\rs[I_k\mid A']$ is monomially equivalent with $\cC$. If $\cC = \rs[I_k\mid A]$ then it follows by Remark \ref{R-CCt} that $\cC^\perp = \rs[-A\T\mid I_{n-k}]$. When $m$ is not a prime the arithmetic depends on the prime factorization of $m$ and thus the notion of a generating matrix and the existence of a generating matrix in ``standard" form becomes less obvious. However, if $m = p^\ell$ is a prime power, it is possible to write a generating matrix in standard form. For the details we refer the reader to \cite{CS95}, and to \cite{NG00} for a more general approach.
\end{rem}
As bypassed in the remark above, working with linear codes over $\Z_m$ (or over any finite ring) requires linear algebra over rings. The standard references on this topic are \cite{McDonaldBernardR1984Laoc,BrownWilliamC1993Mocr}. Although $\Z_m$ has zero divisors in general, many standard results from linear algebra over fields still remain true (like Remark \ref{R-CCt} above), which make the theory of linear codes over rings quite rich. Later on we will make use of the following.
\begin{rem}\label{R-CCt1}
Let $A\in \Z^{n\times n}$ be a square matrix. View $A$ as the $\Z$-module homomorphism $A:\Z_m^n\longrightarrow \Z_m^n, \,x\longmapsto xA$. Then $\ker A, \,\im A\subseteq \Z_m^n$ are linear codes over $\Z_m$. It is straightforward to see that $\Z_m^n = \ker A\oplus \im A$ and $(\ker A)^\perp = \im (A\T)$. Then Remark \ref{R-CCt} implies $|\ker A| = |\ker (A\T)|$ and $|\im A| = |\im (A\T)|$.
\end{rem}
The rate of a linear code $\cC$, denoted $\rate(\cC)$, is the ratio $\log_m(|\cC|)/n$. If $\cC$ is free of rank $k$ then we of course have $\log_m(|\cC|) = k$ and thus $\rate(\cC) = k/n$. Let $\cC_i\subseteq \Z_{m_i}^{n_i}$ be a family of linear codes. We say that the family $\{\cC_i\}$ is \emph{asymptotically good} if 
\begin{equation}
\lim_{i\rightarrow \infty}\rate(\cC_i) \text{ and } \lim_{i\rightarrow \infty}\frac{\dist(\cC_i)}{n_i}
\end{equation}
exist and are nonzero. Otherwise we say that the family $\{\cC_i\}$ is \emph{asymptotically bad}. We refer to the limit values (if they exist) as the \emph{rate of the family $\{\cC_i\}$} and as the \emph{minimum distance of the family $\{\cC_i\}$} respectively. We point out here that the above notions differ from standard asymptotic considerations in the sense that the alphabet size is not constant. 

The following upper bound on the minimum distance, called the Singleton Bound, is well-known \cite{MS77}:
\begin{equation}\label{e-MDS}
\dist(\cC)\leq n - \log_m(|\cC|) +1.
\end{equation}
If \eqref{e-MDS} holds with equality then $\cC$ is called a \emph{Maximum Distance Separable} (MDS) code. We will be mostly focusing on the case $m =p$ prime for which we have the following practical method to compute the minimum distance.
\begin{theo}\label{T-dist}
Let $H$ be a parity-check matrix of a $k$-dimensional linear code $\cC\subseteq \Z_p^n$. Then $\dist(\cC) = d$ iff the minimal number of linearly dependent columns of $H$ is $d$. In particular, $\cC$ is MDS iff $H$ has $n-k+1$ linearly dependent columns and every $n-k$ columns of $H$ are linearly independent.
\end{theo}


\section{Reflexive Laplacian Simplices}\label{sec:Laplacian}
Let $G$ be a simple connected graph with vertex set $V(G):=[n]=\{1,\ldots,n\}$ and edge set $E(G)$. 
The \emph{Laplacian matrix} of $G$, denoted $L_G$, is the difference of the degree matrix and the $\{0, 1\}$-adjacency matrix of $G$.
Consequently, $L_G$ has rows and columns indexed by $[n]$ with entries $a_{ii}=\deg{i}$, $a_{ij} = -1$ if $\{i,j\} \in E(G)$, and $0$ otherwise.
Let $\tau(G)$ denote the number of spanning trees of $G$.
For the cycle graph with $n$ vertices, $\tau(C_n) = n$.
For the complete graph with $n$ vertices, $\tau(K_n)=n^{n-2}$ using Cayley's formula.
When there is no ambiguity we will denote the Laplacian matrix of $G$ by $L$ and the number of spanning trees by $\tau$.

\begin{notation}We often refer to a submatrix of $L_G$ defined by restricting to specified rows and columns.
For $S, T \subseteq [n]$, define $L_G(S \mid T)$ to be the matrix with rows from $L_{G}$ indexed by $[n]\setminus S$ and columns from $L$ indexed by $[n] \setminus T$. 
For ease of notation, we define $L_{G}(i)$ to be the matrix obtained by deleting the $i$\ss{th} column of $L_{G}$, that is, $L_{G}(i) := L( \oldemptyset \mid i) \in \Z^{n \times (n-1)}$.
Finally, $[L_G(i) \mid \one] \in \Z^{n \times n}$ is the matrix $L_{G}(i)$ with an appended column of $\one$\footnote{The length of $\one$ will be clear from context. Note that, with a slight abuse of notation, we use $\one$ both as a row and as a column vector.}. 
\end{notation}

The following properties of the Laplacian matrix are well-known. For details we refer the reader to \cite{Bapat}.

\begin{theo}\label{T-L}
The Laplacian matrix $L_G$ of a connected graph $G$ with vertex set $[n]$ satisfies the following:
\begin{arabiclist}
\item $L_G \in \Z^{n \times n}$ is symmetric, and thus diagonalizable.
\item Each row and column sum of $L_G$ is $0$. In particular, $\rk(L_G) = n-1$. 
\item $\ker_{\R}{L_G} = \langle \one \rangle$ and $\text{im}_{\R}L_G=\langle \one \rangle^{\perp}$, where we view $L_G$ as the linear map $\vect{x} \longmapsto \vect{x}L$ and the perp denotes the standard orthogonal complement.
\item(\mbox{The Matrix-tree Theorem~\cite{Kirchhoff}}). $\tau(G)$ equals the product of nonzero eigenvalues of $L_G$ divided by $n$. As a consequence, any cofactor of $L_G$ is equal to $\tau(G)$.
\end{arabiclist}
\end{theo}

\begin{rem}\label{R-rk}
Theorem \ref{T-L}(4) shows that the $(i,n)$-cofactor of $[L(n) \mid \one]$ for each $i \in [n]$ is equal to the $(i,n)$-cofactor of $L$, which is $\tau(G)$.
Consequently, $\det{[L(n) \mid \one]} = \sum_{i=1}^n C_{in} = n\tau(G)$ where $C_{in}$ is the $(i,n)$-cofactor of $[L(n) \mid \one]$.
\end{rem}


\begin{defiprop}[\cite{BM17}]
Let $G$ be a simple connected graph. The lattice polytope $\cP_{G,\,i}:= \conv{L(i)}$ is a simplex in $\R^{n-1}$. For any $i \neq j$ we have that $\cP_{G,\,i}$ and $\cP_{G,\,j}$ are unimodularly equivalent. We call $\cP_G:=\cP_{G,\,n}$ the \emph{Laplacian simplex associated to the graph $G$}.
\end{defiprop}
\begin{proof}
By Theorem \ref{T-L}(3) we have that $\cP_{G,\,i}$ is indeed a simplex in $\R^{n-1}$. By making use of Theorem \ref{T-L}(2) it is easy to find $U\in \GL_{n-1}(\Z)$ that yields the equivalence of $\cP_{G,\,i}$ and $\cP_{G,\,j}$.
\end{proof}

\begin{theo}[\cite{BM17}]\label{properties}
The Laplacian simplex $\cP_G$ satisfies the following properties.
\begin{arabiclist}
\item $\cP_G$ has normalized volume equal to $n\tau(G)$.
\item $\cP_G$ contains the origin in its interior.
\item For each $0 \le k \le n-1$ we have $\frac{k}{n}\one \in \Lambda(\cP_G)$, which in turn implies $h_i^*(\cP_G) \ge 1$ for all $0 \le i \le n-1.$
\end{arabiclist}
\end{theo}


\begin{theo}\label{fpp}
Let $G$ be a simple connected graph on $n$ vertices with $\tau$ spanning trees. Then
\begin{equation}\label{e-lpg}
\Lambda(\cP_G) = \left\{ \frac{\vect{x}}{n\tau} \,\middle| \, \vect{x} \in \Z^n, \,  0 \le x_i < n\tau, \, \vect{x} \cdot [L(n) \mid \one] \equiv \vect{0} \bmod n\tau \right\}.
\end{equation}
\end{theo}

\begin{proof}
We show the forward containment. The reverse containment is similar. Let $\lambda\in\Lambda(\cP_G)$. Then $\lambda \cdot [L(n) \mid \one] = :\vect{p}_{\lambda} \in \Z^n$.
Cramer's rule implies $\lambda$ takes the form
\begin{equation}\label{e-bi}
1>\lambda_i = \frac{x_i}{\det{[L(n) \mid \one]}} = \frac{x_i}{n\tau},
\end{equation}
where $x_i = \det{([L(n) \mid \one](i, \vect{p}_{\lambda}))} \in \Z_{\ge 0}$ and $[L(n) \mid \one](i, \vect{p}_{\lambda})$ is the matrix obtained by replacing the $i$\ss{th} row of $[L(n) \mid \one]$ by $\vect{p}_{\lambda}$. Put $\vect{x} = (x_1,\ldots,x_n)$. Then \eqref{e-bi} implies $0\leq x_i<n\tau$ as well as $\vect{x} \cdot [L(n) \mid \one] \equiv \vect{0} \bmod n\tau$.
\end{proof}

 


\begin{rem}\label{R-ker}
Theorem~\ref{fpp} constitutes a useful tool for computing $\Lambda(\cP_G)$. Indeed, combining \eqref{e-lfp} and Theorem \ref{properties}(1) we have $|\Lambda(\cP_G)|=n\tau$. Thus, all we need to do for determining $\Lambda(\cP_G)$ is to find the $n\tau$-many vectors \vect{x} that satisfy $0 \le x_i < n\tau$ and 
\begin{equation}
\vect{x} \cdot [L(n) \mid \one] \equiv \vect{0} \bmod n\tau\,.
\end{equation}
Now consider $[L(n)\mid \one]\in \Z^{n\times n}$ as a $\Z$-module homomorphism $\Z_{n\tau}^n\longrightarrow \Z_{n\tau}^n$ as in Remark \ref{R-CCt1}. Then, Theorem \ref{fpp} implies
\begin{equation}\label{e-ker}
\ker{[L(n) \mid \one]} = \left\{ \ov{\vect{x}} \,\,\middle|\,\, \frac{\vect{x}}{n\tau} \in \Lambda(\cP_G)\right\}.
\end{equation}
In particular we have $|\ker{[L(n) \mid \one]}| = n\tau$, which is not a priori clear. Indeed, since $\det[L(n)\mid\one] = n\tau \equiv 0\mod n\tau$, all we know is that $|\ker [L(n)\mid \one]|\geq n\tau$. Recall also from Remark \ref{R-rk} that the $(i,n)$-cofactors of $[L(n)\mid\one]$ are $\tau$, and of course $\ov{\tau}$ is a zero divisor in $\Z_{n\tau}$. On the other hand, there is no other information about the rest of the cofactors. To wrap it up, \eqref{e-ker} says that $\Lambda(\cP_G)$ is completely determined by $\ker[L(n)\mid \one]$. As mentioned in Remark \ref{R-CCt1}, $\ker[L(n)\mid \one]$ is a linear code over $\Zkn$.
\end{rem}
\begin{rem}\label{R-uni}
Let $\cP = \conv{V}$ and $\cP'=\conv{V'}$ be the vertex description of two Laplacian simplices in $\R^{n-1}$. By Theorem \ref{properties}(2) they contain the origin in their interior. Thus $\cP$ and $\cP'$ are unimodularly equivalent iff exists $U\in \GL_{n-1}(\Z)$ such that $V' = VU$. Making use of Remark \ref{R-ker} and the argument therein with the kernel of the $\Z$-linear map $[L(n)\mid \one]$, it follows\footnote{Note that $U$ is invertible modulo $n$ for any $n$, and thus $\ker[V'\mid \one] = \ker [V\mid \one]$.} that $\Lambda(\cP) = \Lambda(\cP')$. In particular, if $L$ is the Laplacian matrix of $G$ then $\Lambda(\cP_{G,\,i}) = \Lambda(\cP_G)$ and $\ker[L(i)\mid \one] = \ker[L(n)\mid \one]$ for any $i\in [n]$. In other words, $\Lambda(\cP_G)$ is an invariant of the equivalence class of $\cP_G$. It follows that for any $\lambda\in \Lambda(\cP_G)$ we have $\lambda \cdot L\in \Z^n$.
\end{rem}

For the remainder of the paper, we narrow our focus to reflexive Laplacian simplices. 
This restriction forces the associated linear code to be duality-preserving. 
Before we define the linear code, however, we establish properties of reflexive $\cP_G$.

\begin{theo}\label{T-reflexive}
Let $\cP_G$ be the Laplacian simplex associated to $G$. Let $L:=L_G$ be its Laplacian matrix. The vertex matrix $V$ of the dual polytope $\cP_G^{\w}$ satisfies $-\tau V = C(n)$,
where $C$ is the cofactor matrix of $[L(n) \mid \one]$. In addition, the $n^{\text{th}}$ column of $C$ is $\tau\one$.
\end{theo}
With a slight abuse of notation, we will write $\pm\tau[V\mid \one] = C$ to have lighter notation for future use. At any rate, the sign will be irrelevant for our purposes. 

\noindent\textit{Proof of Theorem} \ref{T-reflexive}. Recall from Remark~\ref{R-dual} the hyperplane description of the dual simplex is given by $\cP_{G}^{\w} = \{\vect{x} \in \R^{n-1} \mid L(n) \vect{x}\T \le \one \}.$
Since the first minors of $L(n)$ are nonzero, each vertex of $\cP_G^{\w}$ is the intersection of $(n-1)$ hyperplanes. For each $i \in [n]$, let $\vect{v}_i$ be the vertex of $\cP_G^{\w}$ which satisfies $L(i \mid n) \vect{v}_i\T = \one.$ Reindex the rows of $L(i \mid n)$ in increasing order by $[n-1]$. 
Then 
\begin{equation*}
\begin{split}
\vect{v}_i = L(i \mid n)^{-1} \cdot \one 
	&= \frac{1}{\det{L(i \mid n)}} \bar{C}\T \cdot \one = \frac{(-1)^{i+n}}{\tau} \left(\sum_{k=1}^{n-1} \bar{C}_{k1}, \sum_{k=1}^{n-1} \bar{C}_{k2}, \ldots, \sum_{k=1}^{n-1} \bar{C}_{k(n-1)}\right),
\end{split}
\end{equation*}
where $\bar{C}$ is the cofactor matrix of $L(i \mid n)$. These $\{\vect{v}_i\}_{i=1}^n$ form the rows of the matrix $V$. Then the $ij\ss{th}$ entry of $\tau V$ is 
\[
\tau V_{ij} = (-1)^{i+n} \sum_{k=1}^{n-1} (-1)^{k+j} \det{L(i,k \mid j, n)}.
\]
Now consider the cofactor matrix $C$ of $[L(n) \mid \one]$.  
For $j \ne n$, we have 
\begin{equation*}
\begin{split}
C_{ij} &= (-1)^{i+j}\det{([L(n) \mid \one](i \mid j))} \\
&= (-1)^{i+j} \sum_{k=1}^{n-1} (-1)^{k+(n-1)}\det{L(i,k \mid j,n)} \\
&= -\tau V_{ij}\,,
\end{split}
\end{equation*}
where $\det{([L(n) \mid \one](i \mid j))}$ is computed by cofactor expansion along the last column of $\one$.
For $j = n$, we have $C_{in} = (-1)^{i+n}\det{L(i \mid n)} = \tau$ by the Matrix-tree Theorem.  
\QED

Applying the above result to the very definition of a reflexive polytope we obtain the following; see also \cite[Thm. 3.7]{BM17}.
\begin{cor}\label{C-ref}
The Laplacian simplex $\cP_G$ is reflexive iff $\tau$ divides every cofactor of $[L(n) \mid \one]$.  
\end{cor}

This characterization of reflexive Laplacian simplices yields an improvement of the description of $\Lambda(\cP_G)$ from Theorem~\ref{fpp}.

\begin{theo}\label{T-kappadivides}
Let $G$ be a simple connected graph with $n$ vertices such that the associated $\cP_G$ is reflexive. Then 
\[
\Lambda(\cP_G) = \left\{ \frac{\vect{x}}{n} \, \middle| \, \vect{x} \in \Z^n, \,  0 \le x_i < n, \, \vect{x} \cdot [L(n) \mid \one] \equiv \vect{0} \bmod n \right\}.
\]
\end{theo}

\begin{proof}
For $\lambda \in \Lambda(\cP_G)$, we have $\lambda \cdot [L(n) \mid \one] =: \vect{p}_{\lambda} \in \Z^n$. 
Then 
\begin{equation}\label{e-pct}
\lambda = \vect{p}_{\lambda} \cdot [L(n) \mid \one]^{-1} = \frac{\vect{p}_{\lambda}}{n\tau}  \cdot C\T
\end{equation}
where $C\T$ is the transpose of the cofactor matrix of $[L(n) \mid \one]$. 
By Theorem~\ref{T-reflexive}, each entry of $C$ is a multiple of $\tau$.
Thus $\tau$ cancels out in~\eqref{e-pct}.
\end{proof}

Theorem \ref{T-kappadivides} enables us to associate to a reflexive $\cP_G$ a duality-preserving linear code.

\begin{defi}
Let $G$ be a simple connected graph with $n$ vertices such that the associated $\cP_G$ is reflexive. The submodule 
\[
\cC(\cP_G):=\ker [L(n)\mid \one] = \left\{\,\ov{\vect{x}}\mid \vect{x}/n\in \Lambda(\cP_G)\right\}\subseteq \Z_n^n
\]
is called the \emph{linear code associated to the reflexive Laplacian simplex $\cP_G$}.
\end{defi}

\begin{rem}\label{R-why}
Note that one can associate a linear code over $\Z_{n\tau}$ of length $n$ and cardinality $n\tau$ to any Laplacian simplex $\cP_G$. However, in this case the linear code will have an extremely low rate of $1/n\tau$. If $\cP_G$ is reflexive then $\rate(\cC(\cP_G)) = \log_n(n\tau)/n$. Thus graphs with a high number of spanning trees and reflexive associated Laplacian simplices yield linear codes with high rate. For instance, the Laplacian simplex associated to the complete graph on $n$ vertices $K_n$ is reflexive by Theorem~\ref{T-complete}. Since $\tau(K_n) = n^{n-2}$ we conclude that 
\[
\lim_{n\rightarrow \infty}\rate(\cC(\cP_{K_n})) = \lim_{n\rightarrow \infty}  \frac{n-1}{n} = 1.
\]
However, as we will see, $\dist(\cC(\cP_{K_n}))=2$, and thus the family $\{\cC(\cP_{K_n})\}$ is asymptotically bad. Excessive spanning trees of course will yield a low minimum distance. A similar scenario holds for graphs with a low number of spanning trees. The families $\{\cC(\cP_{T_n})\}$ and $\{\cC(\cP_{C_n})\}$ associated to trees and cycles are asymptotically bad because, as we will see, though $\dist(\cC(\cP_{T_n})) = n$ and $\dist(\cC(\cP_{C_n})) = n - 1$ we have $\rate(\cC(\cP_{T_n})) = 1/n$ and $\rate(\cC(\cP_{C_n})) = 2/n$. So the ultimate goal must be to keep a balance between the number of spanning trees and the minimum distance.
\end{rem}
In addition to the reasons mentioned in Remark \ref{R-why}, reflexive Laplacian simplices produce a nice duality relation, which we discuss next. Note first that the dual of a reflexive Laplacian simplex, though reflexive, is not Laplacian in general. See Theorem~\ref{T-completedual} for an example where $\cP^{\w}_G$ is a Laplacian simplex. Despite this, we have the following.

\begin{theo}\label{T-dualthm}Let $G$ be a simple connected graph with $n$ vertices such that the associated $\cP_G$ is reflexive. Then 
\begin{equation}\label{e-2}
\Lambda(\cP^{\w}_G) = \left\{\frac{\vect{x}}{n} \,\,\middle|\,\, \ov{\vect{x}} \in \cC(\cP_G)^{\perp}\right\}.
\end{equation}
\end{theo}

\begin{proof}
Note first by Remark \ref{R-CCt1} we have $\cC(\cP_G)^{\perp} = (\ker[L(n)\mid \one])^\perp = \im ([L(n)\mid\one]\T)$. Thus $\ov{\vect{x}}\in \cC(\cP_G)^{\perp}$ iff there exists $\ov{\vect{z}}\in \Z_n^n$ such that $\ov{\vect{x}} = \ov{\vect{z}}[L(n)\mid\one]\T$. As usual, let $C$ be the cofactor matrix of $[L(n)\mid \one]$. As such, $C$ satisfies $[L(n)\mid\one]\T\cdot C = n\tau I_n.$ Thus $\ov{\vect{x}}C = n\tau\ov{\vect{z}}.$ This along with Theorem \ref{T-reflexive} yield
\begin{equation}\label{e-1}
\frac{\vect{x}}{n} [V\mid \one] = \pm\frac{\vect{x}}{n\tau}C \in \Z^n,
\end{equation}
where $V$ is the vertex matrix of $\cP_G^{\w}$. We have shown $``\!\!\supseteq$'' in \eqref{e-2}. Next, by Remark \ref{R-CCt} we have $|\cC(\cP_G)^\perp| =n^n/|\cC(\cP_G)| = n^{n-1}/\tau $. Thus
\[
|\Lambda(\cP_G^{\w})| = |\det([V\mid \one])| = |\det\left(\pm\frac{1}{\tau}C\right)| = \frac{1}{\tau^n}\cdot \frac{(n\tau)^n}{n\tau} = |\cC(\cP_G)^\perp|\,,
\]
and equality in \eqref{e-2} follows.
\end{proof}
\begin{rem}
Note the proof of Theorem \ref{T-dualthm} shows $\Vol(\cP_G^{\w}) = n^{n-1}/\tau$ regardless of whether $\cP_G$ is reflexive or not. 
Once again we focus on reflexive $\cP_G$ in light of the duality \eqref{e-2}.
\end{rem}
Throughout the paper we have been dealing with finite abelian groups, and thus it is worth mentioning how their rich duality theory relates to our specific case. Let $A$ be a finite abelian group. Then $\wh{A}:=\hom_{\Z}(A,\Q/\Z)$ is called the \emph{Pontryagin dual of $A$}. For a subgroup $B\leq A$ we denote $B^\circ:=\{f \in \wh{A}\mid f(b) = 0\text{ for all } b\in B\}.$ It is well-known that $\wh{A}\cong A$ and $B^\circ \cong \wh{A/B}$.
\begin{rem}\label{R-dual1}
Denote $\Lambda_n:=\{\vect{x}/n\mid \ov{\vect{x}}\in \Z_n^n\}$. With addition as in \eqref{e-add}, $\Lambda_n$ becomes a finite abelian group of $(\Q/\Z)^n$. Let $G$ be a simple connected graph with $n$ vertices such that the associated $\cP_G$ is reflexive. Then Theorem \ref{T-kappadivides} implies $\Lambda(\cP_G)\leq \Lambda_n$ is a subgroup. Clearly, the map 
\[
\Z_n^n\longrightarrow \Lambda_n, \,\,\, \ov{\vect{x}} \longmapsto \vect{x}/n,
\]
is an isomorphism of groups that maps $\cC(\cP_G)$ to $\Lambda(\cP_G)$. Then, combining Theorem \ref{T-dualthm} and Remark \ref{R-CCt} we obtain
\[
\Lambda(\cP_G)^\circ \cong \wh{\Lambda_n/\Lambda(\cP_G)} \cong \Z_n^n/\cC(\cP_G) \cong \cC(\cP_G)^\perp \cong \Lambda(\cP_G^{\w}).
\]
\end{rem}

\subsection{Graph Operations}

This section explores graph operations on $G$ which preserve the reflexivity of $\cP_G$.
In~\cite{BM17}, the authors present \emph{whiskering} and \emph{bridging} of graphs that satisfy certain conditions as behaving nicely with reflexive Laplacian simplices.  
Here we lift these conditions to improve their results by considering the group $\Lambda(\cP_G)$ for these constructions on $G$. We also present new graph constructions which preserve the reflexivity of $\cP_G$.

\begin{defi}\label{D-whisker}
To \emph{whisker} a graph $G$ is to attach an edge and vertex to each existing vertex in $G$. 
Let $\cW(G)$ denote whiskered graph of $G$. 
If $V(G) = \{1, \ldots n\}$, then $V(\cW(G)) = V(G) \cup \{1+n, \ldots, n+n\}$ and $E(\cW(G)) = E(G) \cup \{i,i+n \}_{i=1}^n.$
\end{defi}

The simplex $\cP_{\cW(G)}$ satisfies $\dim{\cP_{\cW(G)}} = 2\dim{\cP_G},$ since we are doubling the number of vertices. It is known \cite[Prop. 5.6]{BM17} that if $\cP_G$ is reflexive, then $\cP_{\cW(G)}$ is reflexive. The structure of $\Lambda(\cP_{\cW(G)})$ in terms of $\Lambda(\cP_G)$ is given below.

\begin{theo}\label{T-whisker}
Let $G$ be a graph with vertex set $[n]$ such that $\cP_G$ is reflexive. Let $L \in \Z^{2n \times 2n}$ be the Laplacian matrix of $\cW(G)$. Then the lattice points in $\Pi(\cP_{\cW(G)})$ are of the form
\begin{equation}\label{e-whisker}
\left\{ (\lambda, \lambda) \cdot [L(n) \mid \one], \left( (\lambda, \lambda) + \frac{1}{2n} \one \right) \cdot [L(n)\mid \one]\,\, \middle|\,\, \lambda \in \Lambda(\cP_G) \right\}.
\end{equation}
Here the $n\ss{th}$ column of $L$ is deleted to make use of $L_G(n)$.
\end{theo}

\begin{proof}
The Laplacian matrix of $\cW(G)$ is of the form
\[
L := \left(\begin{array}{c|c}
L_G+I_n & -I_n \\
\hline
-I_n & I_n
\end{array}\right),
\]
where $L_G$ is the Laplacian matrix of $G$.
When considering $\lambda \cdot [L(n) \mid \one] \in \Z^{2n}$, the last $n$ columns of $L(n)$ yield $\lambda_i = \lambda_{i+n}$ for each $i \in [n]$.
Now let $\lambda \in \Lambda(\cP_G)$.
If $\lambda \cdot L_G(n) =: \vect{p}_\lambda \in \Z^{n-1}$, then $(\lambda,\lambda)\cdot[L(n)\mid \one] = (\vect{p}_\lambda,\vect{0}^n,2 \, \he(\lambda)) \in \Z^{2n},$ which is in $\Pi(\cP_{\cW(G)})$.
Further, $(\lambda, \lambda) + \frac{1}{2n} \one \cdot [L(n) \mid \one] = (\vect{p}_{\lambda}, \vect{0}^n, 2 \, \he(\lambda) + 1) \in \Z^{2n}.$
This is in $\Pi(\cP_{\cW(G)})$ because $(\lambda, \lambda) + \frac{1}{2n} \one$ has entries $\lambda_i$ which satisfy $0 \le \lambda_i \le \frac{n-1}{n} + \frac{1}{2n} = \frac{2n-1}{n} < 1.$

It is left to show \eqref{e-whisker} contains all the lattice points in $\Pi(\cP_{\cW(G)})$.
Let $\tau(G)$ be the number of spanning trees of $G$. 
Then $|\{ (\lambda, \lambda) \mid \lambda \in \Lambda(\cP_G)\}| = n\tau(G)$, which implies \eqref{e-whisker} contains $2n\tau(G)$ lattice points.  
This equals $|\Pi(\cP_{\cW(G)}) \cap \Z^{2n}| = 2n\tau(\cW(G))$, as computed from Theorem~\ref{properties}(1) along with the observation $\tau(G) = \tau(\cW(G))$.
\end{proof}

\begin{cor}
The Laplacian simplex $\cP_{\cW(G)}$ such that $\cP_G$ is reflexive has $h^*$-polynomial of the form
\[
h^*_{\cP_{\cW(G)}}(z) = (1+z)h^*_{\cP_G}(z^2)
\]
where $h^*_{\cP_G}(z)$ is the $h^*$-polynomial of $\cP_G$.
\end{cor}
\begin{proof}
Recall from Lemma~\ref{lem:fpp}, $h^*_{j}(\cP_{\cW(G)}) = |\Pi(\cP_{\cW(G)}) \cap \{ \vect{p} \in \Z^{2n} \mid p_{2n} = j\}|.$
Then for each $i$, \, $0 \le i \le n-1$, 
\[
h^*_{2i}(\cP_{\cW(G)}) = |\{ (\lambda, \lambda) \mid \lambda \in \Lambda(\cP_G), \he(\lambda)=i \}| = h^*_{2i+1}(\cP_{\cW(G)}) 
\]
by Theorem~\ref{T-whisker}.
Then the $h^*$-polynomial of the Laplacian simplex associated to the whiskered graph is $h^*_{\cP_{\cW(G)}}(z) = h^*_{\cP_G}(z^2) + zh^*_{\cP_G}(z^{2}).$
\end{proof}

\begin{cor}
$h^*(\cP_{\cW(G)})$ is unimodal if and only if $h^*(\cP_G)$ is unimodal.
\end{cor}

\begin{rem}\label{R-whisker}
The result in Theorem~\ref{T-whisker} can extend to reveal further families of reflexive simplices. Let $\cW_2(G)$ be the graph obtained from $\cW(G)$ by whiskering the whiskers. In general, let $\cW_k(G)$ be the graph on $(k+1)n$ vertices obtained from $\cW_{k-1}(G)$.
Notice $\cW_1(G) = \cW(G)$ and $\tau(\cW_k(G))=\tau(G)$ for $k\ge 1$. The following additional results are straightforward to show. 
\begin{arabiclist}
\item If $\cP_G$ is reflexive, then $\cP_{\cW_k(G)}$ is reflexive for $k \ge 1$.
\item For reflexive $\cP_G$, the lattice points in $\Pi(\cP_{\cW_k(G)})$ arise from  
\[
\Lambda(\cP_{\cW_k(G)}) = \left\{ (\underbrace{\lambda,\lambda, \ldots, \lambda}_{k+1}) + \frac{i}{(k+1)n}\one \,\,\middle|\,\, \lambda \in \Lambda(\cP_G),\, i=0,1,\ldots,k \right\}.
\]
\item If $\cP_G$ is reflexive, then $\cP_{\cW_k(G)}$ has $h^*$-polynomial of the form
\[
h^*_{\cP_{\cW_k(G)}}(z) = h^*_{\cP_G}(z^{k+1}) \sum_{d=0}^k z^d.  
\] 
\item $h^*(\cP_{\cW_k(G)})$ is unimodal iff $h^*(\cP_G)$ is unimodal.
\end{arabiclist}
\end{rem}

\begin{rem}Although $\cP_{C_n}$ is not reflexive for even $n$~\cite[Thm. 5.1]{BM17}, $\cP_{\cW(C_n)}$ is reflexive~\cite[Prop. 5.4]{BM17}. 
Thus, whiskering $G$ not only preserves the reflexivity of $\cP_G$ but it can also be used to create a graph with a reflexive Laplacian simplex.
\end{rem}

The next family of reflexive Laplacian simplices we consider yield linear codes of odd length.
The associated graphs are an extension of whiskered complete graphs.  

\begin{defi}\label{D-star}
Define $\cW^*(K_n)$ to be the graph obtained from the whiskered complete graph by starring an additional vertex with all the whiskers in $\cW(K_n)$. Thus $|V(\cW^*(K_n))| = 2n+1$ and $|E(\cW^*(K_n))| = |E(\cW(K_n))| + n.$   
\end{defi}

\begin{lem}\label{L-star}
$\tau(\cW^*(K_n))= (2n+1)^{n-1}$.
\end{lem}
\begin{proof}
First, note the Laplacian matrix of the complete graph on $n$ vertices is of the form
\begin{equation}\label{e-complete}
L_{K_n}= n I_n - J_n \in \Z^{n\times n}
\end{equation}
where $J_n$ is the $n\times n$ matrix of all ones.
Set $G:=\cW^*(K_n)$. It is straightforward to see that the Laplacian matrix of $G$ is of the form
\begin{equation}\label{e-Lstar}
L_G := \left(\!\!\begin{array}{c|c|c}
L_{K_n} + I_n & -I_n &\vect{0}\T \\
\hline
-I_n & 2I_n& -\one\\
\hline
\vect{0} & -\one & n
\end{array}\!\!\right)\in \Z^{(2n+1)\times(2n+1)}
\end{equation}
with $L_{K_n}$ defined in~\eqref{e-complete}.
By Theorem~\ref{T-L}(5), $\tau(G)$ equals any cofactor of $L_G$. We compute the $(2n+1,2n+1)$ cofactor. Namely, 
\[
\tau(G) = (-1)^{4n+2}\det\left(\!\!\begin{array}{c|c}
L_{K_n} + I_n & -I_n\\
\hline
-I_n & 2I_n
\end{array}\!\!\right) = 
\det\left(\!\!\begin{array}{c|c}
L_{K_n} & I_n\\
\hline
-I_n & 2I_n
\end{array}\!\!\right) = \det(2L_{K_n} + I_n).
\]
By Theorem \ref{T-L}(1), $2L_{K_n}$ is diagonalizable. In addition, it is well-known that the eigenvalues of $L_{K_n}$ are $1$ with multiplicity 1 and $n$ with multiplicity $n-1$. Thus the eigenvalues of $2L_{K_n}+I_n$ are 0 with multiplicity 1 and $2n+1$ with multiplicity $n-1$. This yields $\tau(\cW^*(K_n))=\det(2L_{K_n}+I_n) = (2n+1)^{n-1}$.
\end{proof}

\begin{theo}\label{T-reflstar}
The Laplacian simplex $\cP_{\cW^*(K_n)}$ is reflexive for all $n \ge 3$.
\end{theo}
\begin{proof}
The claim is $\cP_{\cW^*(K_n)} = \{\vect{x}\in \R^{2n} \mid A \vect{x}\T \le \one \}$, where $A$ is an integer matrix of the form
\[A= \left(\!\!\begin{array}{c|c}
-J_n-2I_n & -I_n \\
\hline
J_n - I_n & J_n - (n\!+\!1)I_n \\
\hline
3\cdot \one & 2 \cdot \one 
\end{array}\!\!\right) \in \Z^{(2n+1) \times 2n}, \]
and $J_n$ is the $n\times n$ matrix of all ones. 
By Remark~\ref{R-hyperplane}, 
it is sufficient to show each vertex of $\cP_{\cW^*(K_n)}=\conv{\vect{v}_1,\ldots,\vect{v}_{2n+1}}$ satisfies $A(i \mid \emptyset)\vect{v}_i\T=\one$ and $A_i \vect{v}_i\T < 1$, where $A_i$ is the $i\ss{th}$ row of $A$.
Recall $\vect{v}_i$ is the $i\ss{th}$ row of $L_G(2n+1)$, where $L_G$ is given above in \eqref{e-Lstar}.
Then we consider
\begin{equation*}
\begin{split}
A \cdot L_G(2n+1)\T &=
\left(\!\! \begin{array}{c|c}
-J_n-2I_n & -I_n \\
\hline
J_n - I_n & J_n - (n\!+\!1)I_n \\
\hline
3\cdot \one & 2\cdot \one 
\end{array}\!\!\right)
\left(\!\!\begin{array}{c|c|c}
L_{K_n} + I_n & -I_n &\vect{0} \\
\hline
-I_n & 2I_n& -\one\\
\end{array}\!\!\right) \\
&=
J_{2n+1} - (2n+1)I_{2n+1}, 
\end{split}
\end{equation*}
which shows the claim.
Since $A$ is the vertex matrix of $\cP^{\w}_{\cW^*(K_n)}$, the result follows. 
\end{proof}
\begin{rem} To guarantee reflexivity, the construction in Definition \ref{D-star} applies specifically to complete graphs. This boils down to the fact that $\cP_{K_n}$ has maximum volume among all Laplacian simplices associated to simple connected graphs with $n$ vertices. For instance, consider $\cP_{C_5}$, which is reflexive by Theorem~\ref{T-oddcycle}. Then one computes
\[h^*(\cP_{\cW^*(C_5)}) = (1, 1, 16, 156, 1491, 3831, 3771, 1176, 126, 1, 1),\]
which in turn implies that $\cP_{\cW^*(C_5)}$ is not reflexive.
\end{rem}


\begin{theo}\label{T-star}
Let $G:= \cW^*(K_n)$.
Then $\frac{1}{2n+1}\vect{x} \in \Lambda(\cP_G)$ iff the following hold:
\begin{arabiclist}
\item $0 \le x_i \le 2n$.
\item For each $n+1\leq i \leq 2n$ there exists $k_i\in \Z$ such that $x_i = (n+1)x_{n-i} - \sum_{j=1}^n x_j - (2n+1)k_i$.
\item For each $j\in [n]$ there exists $m_j\in \Z$ such that $x_{2n+1}=2x_{j+n} - x_j - (2n+1)m_j$.
\end{arabiclist}
\end{theo}
Note that the conditions above come from solving $\vect{x}L_G(2n+1)\equiv \vect{0}$ (where $L_G$ is as in \eqref{e-Lstar}). This system of linear modular equations has $n$ free variables, which we conveniently pick to be the first $n$. That is, $x_i\in \{0,\ldots, 2n\}$ for $i\in [n]$, and for $n+1\leq i \leq 2n+1$, $x_i$ is uniquely determined by the first $n$ of $x_i$'s.

\emph{Proof of Theorem} \ref{T-star}. We show $\vect{x} \cdot [L_G(2n+1) \mid \one] \equiv \vect{0} \bmod (2n+1)$ for $L_G$ defined in \eqref{e-Lstar} iff $\vect{x}$ satisfies the above equations.
Let $\vect{c}_i$ denote the $i\ss{th}$ column of $L_G$.
Fix $\{x_i\}_{i=1}^n \in \{0,1,\ldots,2n\}$.
Then for each $i \in [n],$ $\vect{x} \cdot \vect{c}_i \equiv \vect{0} \bmod (2n+1)$ implies that there exists $\ell_i \in \Z$ such that
\begin{equation}\label{e-star1}
x_{i+n} = (n+1)x_i - \sum_{j=1}^n x_j - (2n+1)\ell_i,
\end{equation}
and $0 \le x_{i+n} \le 2n$.
Also for each $i \in [n], \vect{x} \cdot \vect{c}_{i+n} \equiv \vect{0} \bmod (2n+1)$ implies there exists $m_i \in \Z$ such that
\begin{equation}\label{e-star2}
x_{2n+1} = 2x_{i+n} - x_i - (2n+1)m_i,
\end{equation}
and $0 \le x_{2n+1} \le 2n$.
Observe $x_{2n+1}$ satisfies  
\begin{equation}\label{e-star3}
x_{2n+1} = (2n+1)x_i - 2 \sum_{j=1}^n x_j - (2n+1)(2\ell_i-m_i) \equiv \left(-2 \sum_{j=1}^n x_j\right) \bmod (2n+1).
\end{equation}
Set $x_{2n+1}$ to be the unique integer that satisfies $0 \le x_{2n+1} \le 2n$ and \eqref{e-star3}.
Now it is left to check $\vect{x} \cdot \one \equiv \vect{0} \bmod (2n+1)$.
We have
\begin{equation*}
\begin{split}
\vect{x} \cdot \one &= \sum_{i=1}^{2n+1} x_i \\
&= \sum_{i=1}^n x_i + \sum_{i=1}^n \left( (n+1)x_i - \sum_{j=1}^n x_j - (2n+1) \ell_i \right) + x_{2n+1} \\
&\equiv (1+n+1-n-2)\sum_{j=1}^n x_j -(2n+1)\sum_{i=1}^n \ell_i \bmod (2n+1)  \\
&\equiv 0 \bmod (2n+1).
\end{split}
\end{equation*}
It is left to show these are all the vectors in $\Lambda(\cP_G)$. 
For each $i \in [n]$, there are $2n+1$ choices for $x_i$.
The coordinates $x_{i+n}$ and $x_{2n+1}$ are uniquely determined by the above equivalences. 
Then the number of $\vect{x}$ that satisfy Theorem~\ref{T-star} is $(2n+1)^n$.
Indeed, using Lemma~\ref{L-star} we see this is equal to $|\Lambda(\cP_G)| = \Vol(\cP_G) = (2n+1)^n$ as computed from Theorem~\ref{properties}(1).\QED
\begin{rem}\label{R-whiskerstar}
Similar to Remark~\ref{R-whiskerstar}, the result in Theorem~\ref{T-star} can extend to reveal further families of reflexive simplices when we consider $\cW^*_k(K_n)$, the graph on $(k+1)n+1$ vertices with $\cW_k(K_n)$ defined in Remark~\ref{R-whiskerstar}. The following can be shown.
\begin{arabiclist}
\item The simplex $\cP_{\cW^*_k(K_n)}$ is reflexive for all $k\ge 1$.
\item $\tau(\cW^*_k(K_n)) = ((k+1)n+1)^{n-1}.$ 
\item The description of $\Lambda(\cP_{\cW^*_k(K_n)})$ is similar to Theorem~\ref{T-star}. Each $\lambda=\frac{\vect{x}}{(k+1)n+1} \in \Lambda(\cP_{\cW^*_k(K_n)})$ is uniquely determined by any choice of the coordinates $x_i \in \{0, \ldots, (k+1)n\}$ for $i \in [n]$. 
\end{arabiclist}
\end{rem}


The next graphical construction we consider is attaching two graphs together with an edge to form a bridge. 

\begin{defi}
Let $G$ and $G'$ be graphs with vertex set $[n]$. The \emph{bridge} between $G$ and $G'$, denoted $\cB(G,G')$, is the graphs with vertex set $V(G) \cup V(G')$ and edge set $E(G) \cup E(G') \cup \{i,j\}$ where $i\in V(G)$ and $j \in V(G')$. 
\end{defi}

The resulting simplex satisfies $\dim{\cP_{\cB(G,G')}}=\dim{\cP_G} + \dim{\cP_{G'}}.$ We consider the bridge of two graphs with the same number of vertices and present a sufficient condition for the bridge to produce a reflexive Laplacian simplex. 

\begin{theo}\label{T-fppbridge}
Let $G'$ and $G''$ be graphs with vertex set $[n]$ such that $\cP_{G'}$ and $\cP_{G''}$ are reflexive. The lattice points in $\Pi(\cP{_{\cB(G',G'')}})$ are of the form
\begin{equation}\label{e-bridge1}
(\lambda', \lambda'') \cdot [L(2n) \mid \one]
\end{equation}
where $\lambda' \in \Lambda(\cP_{G'}), \lambda'' \in \Lambda(\cP_{G''}), \lambda'_n = \lambda''_n$, and of the form 
\begin{equation}\label{e-bridge2}
\left((\lambda', \lambda'') + \frac{1}{2n}\one\right) \cdot [L(2n) \mid \one]
\end{equation}
with $\lambda',\lambda''$ as above.
\end{theo}

\begin{proof}
Label the vertices of $\cB(G',G'')$ with $[2n]$ such that vertex $n \in V(G')$ is bridged with vertex $2n \in V(G'')$.
Let $L', L''$ be the Laplacian matrices of $G',G''$.
The Laplacian matrix of $\cB(G',G'')$ takes the form 
\[
L(2n) =
\left(\!\!\begin{array}{c|c|c}
& & \\
L'(n) & \vect{c} & 0 \\
& & \\
\hline
& & \\
0 & \vect{c} & L''(n) \\
& &
\end{array}\!\!\right)
\]
where $\vect{c}\T \in \Z^{2n}$ is of the form $(\vect{c}', \underbrace{0, \ldots, 0}_{n})\T + (\underbrace{0, \ldots, 0, 1}_{n}, \underbrace{0, \ldots, 0, -1}_{n})\T$ and $\vect{c}'\T \in \Z^n$ is the $n\ss{th}$ column of $L'$. 
The claim is that \eqref{e-bridge1} is of the form
\[
(\lambda',\lambda'') \cdot [L(2n) \mid \one] = \left(\underbrace{\lambda' \cdot L'}_{n}, \, \underbrace{\lambda'' \cdot L''(n)}_{n-1}, \, \sum_{i=1}^{n} (\lambda'_i + \lambda''_i) \right) \in \Pi(\cP_{\cB(G',G'')}) \cap \Z^{2n}.
\]
First observe the last coordinate is an integer since $\lambda' \in \Lambda(\cP_{G'})$ and $\lambda'' \in \Lambda(\cP_{G''})$.
This also shows $\lambda' \cdot L'(n), \, \lambda'' \cdot L''(n) \in \Z^n.$
Finally, $\lambda'_n = \lambda''_n$ implies $(\lambda',\lambda'') \cdot \vect{c} = \lambda' \cdot \vect{c}'.$ This is an integer because $\lambda' \cdot L' \in \Z^n$ from Remark~\ref{R-uni}.  
This shows the claim.

Since the entries of each column of $L(2n)$ sum to $0$, the first $2n-1$ coordinates of \eqref{e-bridge2} are equal to the first $2n-1$ coordinates of \eqref{e-bridge1}, which are integer.
The last coordinate of \eqref{e-bridge2} is $\sum_{i=1}^n \left( \lambda'_i + \lambda''_i + \frac{1}{n} \right) \in \Z$.  
Finally, $\lambda = (\lambda',\lambda'') + \frac{1}{2n} \one \in \Lambda(\cP_{\cB(G',G'')})$ since $\lambda$ satisfies $0 < \lambda_i \le \frac{n-1}{n} + \frac{1}{2n} = \frac{2n-1}{2n}$ for each $i \in [2n]$. 

Let $\tau', \tau''$ be the number of spanning trees of $G', G''$.
The number of lattice points that \eqref{e-bridge1} yields is 
\[
|\{(\lambda',\lambda'') \in \Lambda(\cP_{G'}) \times \Lambda(\cP_{G''}) \mid \lambda'_n = \lambda''_n\}| = \frac{|\Lambda(\cP_{G'})|\cdot |\Lambda(\cP_{G''})|}{n}=n\tau'\tau'',
\]
where the first equality follows from the fact that for any Laplacian simplex $\cP_G$, $\one/n \in \Lambda(\cP_G)$.
Moreover, \eqref{e-bridge2} yields $n\tau'\tau''$ additional lattice points as there is a clear bijection between lattice points of \eqref{e-bridge1} and lattice points of \eqref{e-bridge2}. 

It is left to show these are all the lattice points in $\Pi(\cP_{\cB(G',G'')})$.
Indeed, the number of spanning trees of $\cB(G',G'')$ is equal to $\tau' \tau''$.
Thus $|\Pi(\cP_{\cB(G',G'')}) \cap \Z^{2n}| =\Vol(\cP_{\cB(G',G'')})= 2n\tau' \tau''$ from Theorem~\ref{properties}(1).
\end{proof}

\begin{lem}\label{L-bijection}
Let $G$ be a graph with vertex set $[n]$ such that $\cP_G$ is reflexive. 
The map 
\[
f: \Lambda(\cP_G) \longrightarrow \Lambda(\cP_G), \,\lambda \longmapsto \frac{n-1}{n}\one - \lambda,
\] is bijective.
Clearly, $\he(\lambda) =i$ iff $\he(f(\lambda)) = n-1-i$. As an immediate consequence, the restriction $f_i : \{ \lambda \in \Lambda(\cP_G) \mid \he(\lambda) = i \} \longrightarrow \{ \lambda \in \Lambda(\cP_G) \mid \he(\lambda) = n-1-i\}$ is bijective\footnote{The existence of a bijection is clear from Theorem~\ref{T-Hibi}. This specific bijection will be used in later results.} for each $i$, $0 \le i \le n-1$. 
\end{lem}
\begin{proof}
For $\lambda \in \Lambda(\cP_G)$ set $\widetilde{\lambda}:=f(\lambda)$. We only need to show that $\widetilde{\lambda}\in \Lambda(\cP(G)$.
To that end, observe that $\lambda \cdot [L(n) \mid \one] = \vect{p}_{\lambda} \in \Z^n$ implies $\widetilde{\lambda} \cdot [L(n) \mid \one] = (0, \ldots, 0, n-1) - \vect{p}_{\lambda} \in \Z^n.$ The statement now follows.
\end{proof}

\begin{theo}\label{T-bridge}
Let $G'$ and $G''$ be graphs with vertex set $[n]$ such that $\cP_{G'}$ and $\cP_{G''}$ are reflexive. Then $\cP_{\cB(G',G'')}$ is reflexive.
\end{theo}

\begin{proof}
Consider the set $\Lambda= \{(\lambda', \lambda'') \in \Lambda(\cP_{G'}) \times \Lambda(\cP_{G''}) \mid \lambda'_n = \lambda''_n \}.$
Let $a_i = |\{ \lambda \in \Lambda \mid \he(\lambda)=i \}|$ for $0 \le i \le 2n-2$.
Then we can write
\begin{equation*}
\begin{split}
a_i &= \sum_{j=0}^i |\{\lambda \in \Lambda \mid \he(\lambda')=j \text{ and } \he(\lambda'')=i-j \}| \\
&= \sum_{j=0}^i |\{ \lambda \in \Lambda \mid \he(\lambda')=n-1-j \text{ and } \he(\lambda'')=n-1-i+j\}|\\
&= a_{2n-2-i},
\end{split}
\end{equation*}
where the first equality follows by Lemma~\ref{L-bijection}.
Then $a_i = a_{2n-2-i}$ for each $i$, $0 \le i < n.$
Let $b_i =|\{\lambda + \frac{1}{2n}\one \mid \lambda \in \Lambda, \he(\lambda) = i \}$ for $i$, $1 \le i \le 2n-1$.
It is easy to see that $b_i = a_{i-1}$. 
In particular we have $b_i = b_{2n-1-i}$ for $i$, $1 \le i < n+1$.
Finally observe $h^*_i(\cP_{\cB(G',G'')}) = a_i + b_i$ for $0 \le i \le 2n-1$ where we define $a_{2n-2}:=0$ and $b_0 := 0$.
Thus $h^*(\cP_{\cB(G',G'')})$ is symmetric, which shows $\cP_{\cB(G',G'')}$ is reflexive.
\end{proof}


\begin{rem}
\begin{arabiclist}
\item In \cite[Thm. 3.14]{BM17} the authors show the same result as Theorem~\ref{T-bridge}  with an additional natural sufficient condition on the second minors of the Laplacian matrix. However, as Theorem \ref{T-bridge} shows, no additional assumption is needed.
\item If one bridges graphs with different vertex set, Theorem \ref{T-bridge} may no longer be true. Indeed, consider $K_3$ and $K_6$. Then the respective Laplacian simplices are reflexive. Yet, one computes $$h^*(\cP_{\cB(K_3,K_6)})=(1,208,1763,7205,12923,9900,2658,333,1),$$ which in turn shows that that the Laplacian simplex associated to the bridge of $K_3$ and $K_6$ is not reflexive. However, graphs having the same vertex set is not necessary, as we will see later on when we consider multiple bridging. 
Thus understanding the reflexivity of the Laplacian of the bridge construction is yet to be understood.
\end{arabiclist}
\end{rem}

\begin{rem}\label{R-bridge3}
The results in Theorems~\ref{T-fppbridge} and~\ref{T-bridge} generalize to the bridge of any number of graphs with the same vertex set. Define $B:=\cB(G_1, G_2, \ldots, G_k)$ be the bridge of $k$ graphs with $V(\cB) = [kn]$ and $E(G) = \bigcup_{j=1}^k E(G_j) \cup \{mn, mn+1 \}_{m=1}^{k-1}$. Observe each vertex is part of at most one bridge, and the order of bridging matters. For instance $\cP_{\cB(K_n, C_n, C_n)}$ is not unimodularly equivalent to $\cP_{\cB(C_n,K_n,C_n)}$ for $n \ge 4$. The following is a straightforward extension of previous results.
\begin{arabiclist}
\item $\Lambda(\cP_B) = \left\{(\lambda_1, \lambda_2, \ldots, \lambda_k) + \frac{i}{kn} \one \in \Q^{kn} \mid i=0,1,\ldots,k-1 \right\}$ where $\lambda_j \in \Lambda(\cP_{G_j})$ for each $j \in [k]$ such that $(\lambda_\ell)_n = (\lambda_{\ell+1})_1$ for each $\ell \in [k-1]$.
\item If $\cP_{G_j}$ is reflexive for each $j\in [k],$ then $\cP_B$ is reflexive.
\end{arabiclist}
\end{rem}

\section{Linear Codes Associated to Laplacian Simplices}\label{sec:codes}

In this section we analyze asymptotic behavior of families of linear codes associated to families of Laplacian simplices. 
Ultimately we focus on simple connected graphs with $n=p$ number of vertices that yield reflexive $\cP_G$.
Before jumping to specific families we point out some general considerations.  
\begin{rem}\label{R-code}
\begin{arabiclist}
\item Let $\cP_{G}$ and $\cP_{G'}$ be two unimodularly equivalent Laplacian simplices. Then by Remark \ref{R-uni} we have $\cC(\cP_G) = \cC(\cP_{G'})$. 
\item By Theorem \ref{properties}(3) we have $\ov{\one}\in \cC(\cP_G)$ for any Laplacian simplex $\cP_G$.
As an extremal case, $\cC(\cP_G) = \group{\ov{\one}}$ iff $G$ is a tree; see also ~\cite[Prop. 4.1]{BM17}. In particular, any tree on a fixed vertex set has the same associated $\cP_G$ up to unimodular equivalence.
\end{arabiclist}
\end{rem}
Let $G, \,G'$ be two simple connected graphs on $n$ vertices. Then $G, \,G'$ are \emph{isomorphic} if there exists a permutation $\sigma\in S_n$ such that $i\in V(G)$ iff $\sigma(i)\in V(G')$ and $(i,j)\in E(G)$ iff $(\sigma(i),\sigma(j))\in E(G')$. It is well known that $G,\,G'$ are isomorphic iff there exists a permutation matrix $P$ such that $L_{G'} = P\T L_GP$. Assume that the nonzero entry of the $n\ss{th}$ column of $P$ is in row $i$. Then it is easy to verify that $L_{G'}(n)=P\T L_G(i)P(i\mid n)$. Note that $P(i\mid n) \in \GL_{n-1}(\Z)$, and thus multiplying on the right doesn't affect $\cC(\cP_G)$ thanks to Remark \ref{R-code}(1). On the other hand, multiplying on the left with $P\T$ permutes the vertices of $\cP_G$. Making use of Remark \ref{R-code}(1) one more time we obtain the following.
\begin{theo}\label{T-PE}
If two simple connected graphs $G,\,G'$ are isomorphic, then $\cC(\cP_G),\, \cC(\cP_{G'})$ are permutation equivalent.
\end{theo}
\begin{rem}
The converse of Theorem \ref{T-PE} does not hold. Indeed, as pointed out in in Remark \ref{R-code}, $\cC(\cP_{T_n}) = \group{\ov{\one}}$ for any tree on $n$ vertices. Yet, clearly, there exists non-isomorphic trees (for $n\geq 4$).
\end{rem}
Recall from~\cite{BM17} the characterization of fundamental parallelepiped points for Laplacian simplices associated to odd cycles.
\begin{theo}[\mbox{\cite[Thm. 5.9]{BM17}}]\label{T-oddcycle}
For odd $n \ge 3$, the lattice points in $\Pi(\cP_{C_n})$ are of the form 
\begin{equation}\label{e-Cn}
\frac{[\alpha \one + \beta (0,1,\ldots,n-1)] \bmod{n}}{n} \cdot [L_{C_n}(n) \mid \one]
\end{equation}
for integers $0 \le \alpha, \beta \le n-1$.
\end{theo}

\begin{rem}\label{R-oddcycle}
Let $n\ge 5$ be odd and consider the cycle $C_n$. It follows by \eqref{e-Cn} that $\cC:=\cC(\cP_{C_n})$ is the linear code generated by $\ov{\one}: = (\ov{1},\ldots, \ov{1})$ and $\ov{\vect{x}}: = (\ov{0},\ov{1},\ldots,\ov{n-1})$. It is easy to see that $\cC$ is cyclic and $\dist(\cC) = n-1$. Since $|\cC| = |\Lambda(\cP_{C_n})| = n^2$ we have that $\cC$ is (trivially) MDS. Moreover, $\rate(\cC) = 2/n$, and thus linear codes associated to cycles of odd length $n$ are asymptotically bad.
Now consider $n \ge 5$ prime. 
Since $\one \cdot \vect{x} \equiv 0 \bmod{n}$ and $\vect{x}\cdot \vect{x} \equiv 0 \bmod{n},$
it follows that $\cC\subset \cC^\perp$. 
\end{rem}

\begin{theo}[\mbox{\cite[Thm. 6.6]{BM17}}]\label{T-complete}
For $n \ge 3$, the Laplacian simplex associated to the complete graph $\cP_{K_n}$ is reflexive. Moreover
\begin{equation}\label{e-Kn}
\Lambda(\cP_{K_n}) = \left\{ \frac{\vect{x}}{n} \, \middle| \, \vect{x} \in \Z^n, \,  0 \le x_i < n, \, \sum_{i=1}^n x_i \equiv 0 \bmod n \right\}.
\end{equation}
\end{theo}
\begin{rem}\label{R-complete}
It follows by \eqref{e-Kn} that 
\begin{equation}\label{e-CKn}
\cC:=\cC(\cP_{K_n}) = \left\{\ov{\vect{x}}\in\Z_n^n \,\,\middle|\,\, \sum_{i=1}^n \ov{x_i} = \ov{0}\right\}.
\end{equation}
Similar to the odd cycle case, we see that $\cC$ is cyclic. It is also easy to see that $\dist(\cC) = 2$.
Next, $|\cC| = |\Lambda(\cP_{K_n})| = n^{n-1}$. 
It follows that $\cC$ is again (trivially) MDS. 
Though linear codes associated to complete graphs have a high (maximal) rate, $\rate(\cC) = (n-1)/n$, the family is asymptotically bad because $\dist(\cC) = 2$. Thanks to Remark \ref{R-code}(2) and \eqref{e-CKn} we have $\cC^\perp \subset \cC$. 
By Remark \ref{R-CCt} it follows that $|\cC^\perp| = n$. 
In particular, $\cC^\perp = \group{\ov{\one}} = \cC(\cP_{T_n})$ where $T_n$ is a tree on $n$ vertices.
\end{rem}
\begin{theo}\label{T-completedual}
For $n \ge 3$, the dual of the Laplacian simplex associated to a complete graph is unimodularly equivalent to the Laplacian simplex associated to a tree, that is, \[\cP^{\w}_{K_n} \cong \cP_{T_n}\]
where $T_n$ is any tree on $n$ vertices.
\end{theo}
\begin{proof}
The claim is that the hyperplane description is $\cP_{K_n} = \{ \vect{x}\in \R^{n-1} \mid [-I_{n-1} \mid \one]\T \vect{x}\T \le \one\}$.
Using the vertex description of $\cP_{K_n}$ found~\eqref{e-complete}, observe $[-I_{n-1} \mid \one]\T \cdot L_{K_n}(n)\T = J_n - nI_n.$
By Remark~\ref{R-hyperplane}, this shows the claim.
Then the vertex description of $\cP^{\w}_{K_n}$ is $[-I_{n-1} \mid \one]\T$, see Remark~\ref{R-dual}. 
Let $S$ be the star graph with vertex set $[n]$ and $|E(S)|=n$; specifically, $\deg{n}=n-1$ and $\deg{i}=1$ for all $i \in [n-1]$. 
Observe that $[-I_{n-1} \mid \one]\T \cdot -I_{n-1} = [I_{n-1} \mid -\one]\T$ is the Laplacian matrix of the graph $S$ with the $n\ss{th}$ column removed. This shows $\cP^{\w}_{K_n}$ is unimodularly equivalent to $\cP_S$.
Since all trees have unimodularly equivalent Laplacian simplices, the result follows.
\end{proof}
Now we shift our focus to graphs with a prime number of vertices whose Laplacian simplex is reflexive. In this case, we use the dimension of the associated code to classify the number of spanning trees of such a graph.
\begin{rem}\label{R-TreePrime}
Let $\cP_G$ be a reflexive Laplacian simplex where $G$ is a simple connected graph with a prime number of vertices $p$. Then $\cC(\cP_G) \subseteq \Z_p^p$ is a vector space, say of dimension $k$. Then $|\cC(\cP_G)|=p^k$. On the other hand $|\cC(\cP_G)| = p\cdot\tau(G)$. It follows that $\tau(G) = p^{k-1}$. To conclude, the number of spanning trees of a simple connected graph on $p$ vertices whose Laplacian simplex is reflexive is a power of $p$; moreover, that power is precisely one less than the dimension of the associated linear code.
\end{rem}
\begin{theo}\label{T-W*MDS}
Let $p = 2n+1$ be a prime and consider $\cW^*(K_n)$ as in Definition \ref{D-star}. Then $\cC := \cC(\cP_{\cW^*(K_n)})$ is MDS.
\end{theo}
\begin{proof}
Note that Lemma \ref{L-star} implies $|\cC| = p^n$, and thus $\dim(\cC)=n$. It follows that we need to show $\dist(\cC)= n+2$. We do so by making use of Theorem \ref{T-dist}, where all the (linear) algebra is done over $\Z_p$. First, by reading off Theorem \ref{T-star} we may find a generating matrix in standard form. Indeed, $\cC$ is the row space of following matrix $n\times p$:
\begin{equation}
\begin{array}{c@{}c}
G=\left(\!\!
\begin{array}{cccc|ccccc}
1&0&\cdots&0&n&2n&\cdots&2n&2n-1\\
0&1&\cdots&0&2n&n&\cdots&2n&2n-1\\
\vdots&\vdots&\ddots&\vdots&\vdots&\vdots&\ddots&\vdots&\vdots\\
0&0&\cdots&1&2n&2n&\cdots&n&2n-1
\end{array}\!\!
\right)\in \Z_p^{n\times p}.
&
\begin{array}{l}
\end{array} \\
\hspace{-.15 in}\hexbrace{2.1cm}{n}\hspace{.1in}\hexbrace{4.3cm}{n+1}
\end{array}
\end{equation}
The entries of $G$ are in $\Z_p$ where we have omitted the bar. As in Remark \ref{R-code} we obtain the  parity-check matrix of $\cC$:
\begin{equation}
H= \left(\!\!\begin{array}{c|c}
nI_n + J_n &  \\
\rule{1.5cm}{0.15mm} & \text{  } I_{n} \text{  } \\
2 \cdot \one  & 
\end{array}\!\!\right)\in \Z_p^{(n+1)\times p},
\end{equation}
where $J_n$ is the matrix of all ones.
Let $h_i$ be the $i\ss{th}$ column of $H$. Then $\{h_1,\ldots,h_{n+1}\}$ are linearly independent over $\Z_p$ since $\{h_1-h_2,h_1-h_3,\ldots, h_1-h_{n+1}\}$ obviously are. Similarly, one checks that every $n+1$ columns of $H$ are linearly independent. In addition $\{h_1,\ldots, h_{n+1},h_{2n+1}\}$ are linearly dependent since $h_1+\cdots+h_{n+1}-h_{2n+1} = 0$. The statement now follows.
\end{proof}
\begin{cor}\label{C-W*good}
Let $p_i$ be prime and put $n_i:=(p_i-1)/2$. The family $\{\cC(\cP_{\cW^*(K_{n_i})})\}$ is asymptotically good.
\end{cor}
\begin{theo}
Let $n$ be odd and let $a \leq b$ be any natural numbers. Then there exists a family of graphs $\{G_n\}$ such that the family $\{\cC(\cP_{G_n})\}$ has rate $a/b$.
\end{theo}
\begin{proof}
Let $G_n$ be the bridge of $b-a$ copies of $C_n$ and $a$ copies of $K_n$. It follows by Remark \ref{R-bridge3} that $\cP_{G_n}$ is reflexive. We claim that the family $\{\cC(\cP_{G_n})\}$ has the desired rate. Note first that $G_n$ has $bn$ vertices and thus $\cC(\cP_{G_n})\subseteq \Z_{bn}^{bn}$. Next, we have mentioned that the number of spanning of the bridge equals the product of the number of the spanning trees of each component graph. Recall also that $\tau(C_n) = n$ and $\tau(K_n) = n^{n-2}$. Thus 
\[
|\cC(\cP_{G_n})| = (bn)(\tau(G_n)) = (bn)( \tau(C_n))^{b-a}(\tau(K_n))^{a}=bn^{a(n-3)+b+1}.
\]
We now compute
\[
\lim_{n\rightarrow \infty}\rate(\cC(\cP_{G_n})) = \lim_{n\rightarrow \infty}\frac{\log_{bn}|\cC(\cP_{G_n})|}{bn} = \lim_{n\rightarrow \infty}\frac{\log_{bn}bn^{a(n-3)+b+1}}{bn} = \frac{a}{b},
\]
which in turn yields the claim.
\end{proof}

\section{Conclusions and Future Research}
We extended the work of \cite{BM17} and focused on families of graphs that yield reflexive Laplacian simplices. 
We found the vertex description of the dual of a Laplacian simplex. 
As a consequence, this allowed us to classify reflexive Laplacian simplices in terms of the cofactor matrix of the Laplacian matrix and the number of spanning trees; see Corollary \ref{C-ref}. In addition, we considered graph operations that preserve reflexivity. 
We followed the line of \cite{BH13} by studying the group of lattice points in the fundamental parallelepiped of a simplex. 
In this way we associated to a reflexive Laplacian simplex a linear code of length $n$ over the integer residue ring $\Z_n:=\Z/n\Z$ where $n$ is the number of vertices of the graph with which we started. 
We paid special attention to the case $n=p$ being a prime and studied the minimum distance, dimension, and asymptotic behaviors of associated linear codes. 
In particular, we constructed an asymptotically good family of linear codes; see Corollary \ref{C-W*good}.

A natural topic to investigate further is how a graph structure affects the equivalence class of the Laplacian simplex.  
All trees on $n$ vertices yield Laplacian simplices in the same equivalence class. Consequently, graphs enhanced by attaching different trees with the same vertex set produce equivalent Laplacian simplices. 
Another graphical operation which preserves the equivalence class of the associated Laplacian simplex is found~\cite[Prop. 3.10]{BM17}.

\begin{prob}
Which simplices in the equivalence class of $\cP_G$ can be realized as Laplacian simplices? How are the underlying graphs related?
\end{prob}

Typically, the dual of a Laplacian simplex is not a Laplacian simplex. The only known instance is $\cP^{\w}_{K_n}\cong \cP_{T_n}$, as recorded in Theorem~\ref{T-completedual}. This rare behavior prohibits us from associating a linear code to the dual of a reflexive Laplacian simplex; see also Theorem \ref{T-dualthm}.

\begin{prob}
Classify graphs $G$ for which $\cP_G$ is reflexive and $\cP^{\w}_G$ is Laplacian simplex.
\end{prob}

Let $C$ denote the cofactor matrix of $[L_G(n)\mid \one]$ for reflexive $\cP_G$. Note that by Theorem~\ref{T-reflexive}, the above is equivalent to the statement $\frac{-1}{\tau}C(n)$ is unimodularly equivalent to the Laplacian matrix of some graph on $n$ vertices.

Unimodality of the $h^*$-vector of a Laplacian simplex remains quite mysterious. Simple connected graphs with at most eight vertices yield Laplacian simplices with unimodal $h^*$-vector. 
Here we present a graph with nine vertices for which the associated Laplacian simplex is reflexive with a nonunimodal $h^*$-vector. Let $G:= \cB(C_3, T_6)$ be the bridge of the cycle with three vertices and any tree with six vertices. Then one computes 
\[
h^*(\cP_G) = (1,3,3,5,3,5,3,3,1),
\]
which in turns says $\cP_G$ is reflexive, and yet, $h^*(\cP_G)$ is not unimodal. 
One can verify the lattice point $(1,-1,0,0,0,0,0,0,3) \in \Pi(\cP_G)$ cannot be written as a linear combination of the three lattice points in $\Pi(\cP_G)$ at height $1$, which have the form $(0,0,\vect{0},1)$, $(1,0,\vect{0},1),$ and $(0,1,\vect{0},1)$ where $\vect{0} = (0,0,0,0,0,0)$. This shows $\cP_G$ is not IDP, and consequently, $\cP_G$ is not a counterexample to the long standing conjecture found in~\cite{hibiohsugiconj}.

\begin{prob}
Classify graphs $G$ for which $h^*(\cP_G)$ is not unimodal. Moreover, do there exist families of graphs $\{G_i\}$ for which $\cP_G$ is reflexive and $h^*(\cP_G)$ is not unimodal?
\end{prob}

In Section \ref{sec:codes} we focus on graphs with a prime number of vertices for which $\cP_G$ is reflexive. As one will note, all the linear codes considered were MDS. This case is facilitated by the fact that the number of spanning trees of a graph $G$ for which $\cP_G$ is reflexive is a prime power; see Remark \ref{R-TreePrime}. When the number of vertices $n$ is not prime, the number of spanning trees is unknown. We formulate the following.

\begin{conjecture}
Let $G$ be a graph with a prime number of vertices such that $\cP_G$ is reflexive. Then $\cC(\cP_G)$ is MDS.
\end{conjecture}

\begin{prob}
Let $\{G_i\}$ be a family of graphs on $n_i$ vertices such that $\cP_{G_i}$ is reflexive. Assume there exists a polynomial $p(x)$ with integer coefficients such that $\tau(G_i) = n_i^{p(n_i)}$. Is $\cC(\cP_{G_i})$ MDS? 
\end{prob}

In Remark~\ref{R-oddcycle} we saw that codes associated to prime cycles are self-orthogonal. In Remark~\ref{R-complete} we saw that codes associated to complete graphs contain their dual. Whereas codes associated to $\cP_{\cW^*(G)}$ satisfy neither of the above.

\begin{prob}
Does there exist a simple connected graph for which $\cP_G$ is reflexive and $\cC(\cP_G)$ is self-dual?
\end{prob}

Note that graphs satisfying the above must have an even number of vertices, $2n$, and $(2n)^{n-1}$ spanning trees. The smallest simple connected graph that satisfies this is the complete bipartite graph on four vertices $K_{2,2}$; however, $\cP_{K_{2,2}}$ is not reflexive. A computer search shows that all the graphs with at most eight vertices that satisfy the above numerical conditions also do not yield reflexive Laplacian simplices.

We end the paper with a promising (in our view) future direction. Recall equation \eqref{e-ht}. Clearly, the left-hand-side encodes information about the $h^*$-vector of a simplex. The right-hand-side encodes the (Hamming) \emph{weight distribution} of codewords. It is well-known that the weight distribution of the codewords of the dual is completely determined by the \emph{MacWilliams identity} \cite{MacWilliams62, MacwilliamsJessie1963ATot}. Since the association of linear codes and reflexive Laplacian simplices is also duality-preserving (Theorem \ref{T-dualthm}), we formulate the following.
\begin{prob}
Can one use the MacWilliams identity to better understand the $h^*$-vector of the dual of a reflexive Laplacian simplex?
\end{prob}
Making use of Remark \ref{R-dual1} it is easy to see that the question above can be asked for any lattice simplex.
\bibliographystyle{abbrv}
\bibliography{LS2}
\end{document}